\documentclass[12pt]{amsart}
\usepackage{amsmath}
\usepackage{amssymb}
\usepackage{amsfonts}
\usepackage[all]{xy}

\theoremstyle{plain}
\newtheorem{thm}{Theorem}

\theoremstyle{definition}
\newtheorem{defn}[thm]{Definition}
\newtheorem{rem}[thm]{Remark}

\usepackage[]{graphicx}

\newcommand{\CC}{\mathbb{C}}

\newcommand{\PP}{\mathbb{P}}
\newcommand{\QQ}{\mathbb{Q}}

\newcommand{\ZZ}{\mathbb{Z}}

\begin{document}
\title[Covers of Elliptic Curves]{Covers of Elliptic Curves with Unique, Totally Ramified Branch Points}
\author{Simon Rubinstein-Salzedo}
\address{Department of Mathematics, Dartmouth College, Hanover, NH 03755}
\email{simon.rubinstein-salzedo@dartmouth.edu}

\begin{abstract}
A well-known and difficult problem in computational number theory and algebraic geometry is to write down equations for branched covers of algebraic curves with specified monodromy type. In this article, we present a technique for computing such covers in the case of covers of elliptic curves with unique, totally ramified branch points.
\end{abstract}
%% maketitle must follow the abstract.
\maketitle                   % Produces the title.

%% If there is not enough space inside the running head
%% for all authors including the title you may provide
%% the leftmark in one of the following three forms:

%% \renewcommand{\leftmark}
%% {First Author: A Short Title}

%% \renewcommand{\leftmark}
%% {First Author and Second Author: A Short Title}

%% \renewcommand{\leftmark}
%% {First Author et al.: A Short Title}

\section{Introduction}

In this article, we present a technique, based on a degeneration idea that Couveignes in \cite{Couv99} used in a slightly different context, to construct covers of elliptic curves with unique, totally ramified branch points. In particular, given an elliptic curve $E$ and an integer $g\ge 2$, we construct a genus-$g$ curve $C$ and a map $f:C\to E$ of degree $2g-1$, so that $f$ is ramified above exactly one point of $E$, and so that the local monodromy above that point is of type a $(2g-1)$-cycle.

The study of branched covers of curves goes back to Riemann, who determined necessary and sufficient conditions on the local monodromies for such covers to exist. However, interest in writing down explicit equations for such covers is more recent, and it was only after Bely\u\i\ in \cite{Belyi79} proved his celebrated theorem and Grothendieck laid out his Esquisse d'un Programme \cite{Groth97} for using such covers to understand the absolute Galois group of $\QQ$ that interest in this subject took off. More recently, based on work of Beckmann \cite{Beck89}, Roberts in \cite{Rob04} has demonstrated that Bely\u\i\ maps $\PP^1\to\PP^1$ can be used in practice to construct number fields with limited ramification.

While the theory in Beckmann's work and Roberts's work applies to more general covers of curves, computations of covers of other curves has been too difficult for practical use for constructing number fields of limited ramification.

On the other hand, inspired by the analogy with Bely\u\i\ maps, work has been done from a mostly topological perspective on branched covers of elliptic curves. This work began with Lochak in \cite{Lochak05} and has continued with work of M\"oller in \cite{Moller05} and of Herrlich and Schmith\"usen (for example, in \cite{HS09}) and their Karlsruhe school.

In Section \ref{origamis}, we present some background on branched covers of elliptic curves ramified at one point (also known as origamis). In Section \ref{family}, we present our result. Finally, in Section \ref{degen}, we present the degeneration technique used to construct the family of covers for the first time. While the technique presented does not guarantee a solution, we suspect that this section will be the most interesting part of this article.

\section{Origamis} \label{origamis}

By the Riemann-Hurwitz formula, an unramified cover of a genus-1 curve must again be a genus-1 curve, which means that such a map is simply a composition of an isogeny of elliptic curves and a translation.

However, if we allow one branched point on our elliptic curve, then there are covers by higher-genus curves. These will be our objects of study.

\begin{defn} Let $E$ be an elliptic curve. An origami is a pair $(C,f)$, where $C$ is a curve and $f:C\to E$ is a map, branched only above one point. \end{defn}

Origamis are so-called because they admit a pictorial interpretation vaguely reminiscent of the eponymous Japanese art form, analogous to that of dessins for Bely\u\i\ maps.

Over $\CC$, any elliptic curve $E$ can be written as $\CC/\Lambda$, for some lattice $\Lambda\subset\CC$. We will find it most helpful to think of $E$ as a fundamental parallelogram for $\Lambda$. The choice of lattice $\Lambda$ or fundamental parallelogram determines the complex structure on $E$. Many of our arguments do not depend on the choice of complex structure; when this happens, we choose to work with the square lattice $\Lambda=\ZZ[i]$, and our fundamental parallelogram of choice will be the square $S$ with vertices 0, 1, $1+i$, and $i$. (The only reason we prefer this parallelogram is that it is easier to draw than are other parallelograms. It should not generally be assumed that we are interested in the special properties of the elliptic curve $\CC/\ZZ[i]$ not enjoyed by other elliptic curves.) Our elliptic curve will then be the square, with opposite edges identified.

Now, consider a disjoint union of $n$ translates of $S$, and identify various edges to form an orientable surface $X$ subject to the following requirements: \begin{enumerate} \item $X$ is connected. \item Every left edge is identified with a unique right edge, and vice versa. \item Every top edge is identified with a unique bottom edge, and vice versa. \end{enumerate} If we remove all the vertices of the $n$ squares, the resulting figure carries the structure of a Riemann surface, obtaining a complex structure whose charts are (slightly enlarged versions of) the original $n$ squares minus the vertices. The resulting Riemann surface $\widetilde{X}$ is then a compact Riemann surface with several punctures. There is a unique way of compactifying $\widetilde{X}$ so that its compactification is a compact Riemann surface; we call this Riemann surface $X$. Furthermore, $X$ admits a degree $n$ map to the elliptic curve $\CC/\ZZ[i]$ by mapping a point in any translate of $S$ to the corresponding point in $S$. This map is branched only above the vertex of $S$. An example can be seen in Figure \ref{Lorigami}. In this diagram, we have explained the edge identification; in the future, if there are no markings on the edges, we take this to mean that opposite edges are identified. (This will be the case in all origami diagrams in this article.) The map is shown in Figure \ref{origmap}.

\begin{figure}  \begin{center} \includegraphics[height=1in]{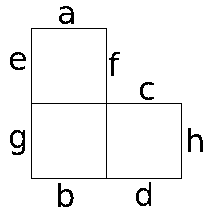} \end{center} \caption[An origami diagram]{This diagram represents a genus-2 curve with a degree-3 map to the elliptic curve $y^2=x^3-x$. Here we identify opposite edges, meaning that edge $a$ is identified with edge $b$, edge $c$ with edge $d$, edge $e$ with edge $f$, and edge $g$ with edge $h$.} \label{Lorigami} \end{figure}

\begin{figure} \begin{center} \includegraphics[height=1in]{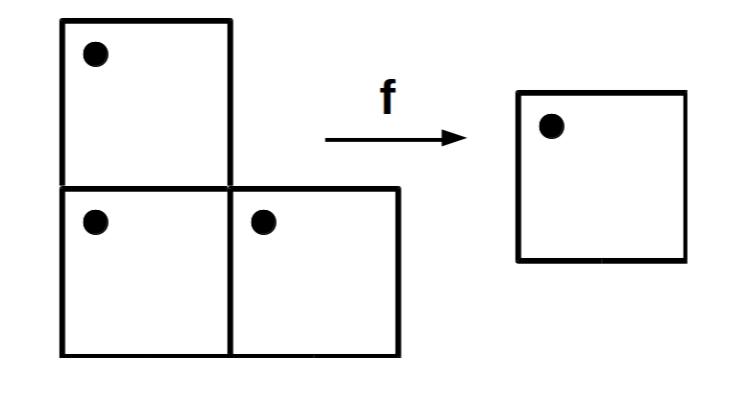} \end{center} \caption{Shown are all the preimages under $f$ in $C$ of the marked point in $E$.} \label{origmap} \end{figure}

%\begin{vchfigure} \includegraphics[height=1in]{L-origami.png} \vchcaption{This diagram represents a genus-2 curve with a degree-3 map to the elliptic curve $y^2=x^3-x$. Here we identify opposite edges, meaning that edge $a$ is identified with edge $b$, edge $c$ with edge $d$, edge $e$ with edge $f$, and edge $g$ with edge $h$.} \label{Lorigami} \end{vchfigure}

%\begin{vchfigure} \includegraphics[height=1in]{origmap.png} \vchcaption{Shown are all the preimages under $f$ in $C$ of the marked point in $E$.} \label{origmap} \end{vchfigure}

Had we chosen to distinguish a different elliptic curve with a different fundamental parallelogram $P$, the corresponding origami would simply consist of a disjoint union of $n$ translates of $P$ with similar edge identifications.

The origami diagram, though apparently extremely simple, turns out to carry a wealth of combinatorial information in readily available form. For example, we can compute the local monodromy about the branch point. To do this, number the squares of $C$ from 1 to the degree $n$ of the map, in any way. We now define two permutations $g,h\in S_n$, which will be the monodromies around two loops generating $H_1(E)$. Let $g$ be the permutation obtained by moving one square to the right, and let $h$ be the permutation obtained by moving one square up. For example, if we label the square in Figure \ref{Lorigami} with a ``1'' in the top left corner, a ``2'' in the bottom left, and a ``3'' in the bottom right, then $g=(23)$ and $h=(12)$. The local monodromy above the branch point is the commutator $[g,h]=g^{-1}h^{-1}gh$, which in this case is the 3-cycle $(132)$. If we relabel the squares, we obtain different permutations $g'$ and $h'$ and a different commutator; however, there is some $\sigma\in S_n$ so that $g'=\sigma^{-1}g\sigma$ and $h'=\sigma^{-1}h\sigma$, so the cycle type of the local monodromy is well-defined.

It is also possible to determine the genus of $C$ from the origami diagram. To do this, we use Euler's formula $V-E+F=2-2g$. The number of faces $F$ is equal to $n$, and the number of edges is $2n$. To determine the number of vertices, we can either check directly which vertices in the diagram are glued to which other vertices, or we can note that the number of vertices is equal to the number of cycles in the cycle decomposition of the local monodromy. Hence, in this case, there is one vertex, so $V-E+F=-2$, and so the genus is 2. In this article, we will only be interested in origamis for which $V=1$.

\section{A family of algebraic origamis} \label{family}

In this section, we construct a family of examples of explicit origamis, one for each genus $g$.

\begin{defn} We say that an origami is totally ramified if the preimage of the branch point is a single point. \end{defn}

The origamis we construct here will all be totally ramified.

\begin{thm} \label{totramorigami} For each $g\ge 1$ and $t\neq 0,-1$, the genus-$g$ curve \[C_t:y^2=x(x+1)(x^{2g-1}+tj(x)^2),\] where \[j(x)=\sum_{i=0}^{g-1} \binom{2g-1}{2i}(x+1)^i,\] admits a degree $2g-1$ map to the elliptic curve \[E_t:y^2=x(x+1)(x+t),\] totally ramified above $(0,0)$ and unramified everywhere else. The map is given by $(x,y)\mapsto(f_1(x),f_2(x)y)$, where \[f_1(x)=\frac{x^{2g-1}}{j(x)^2}\] and \[f_2(x)=\frac{x^{g-1}\sum_{i=0}^{g-1}\binom{2g-1}{2i+1}(x+1)^i}{j(x)^3}.\] \end{thm}

\begin{proof} We first check that $(f_1(x),f_2(x)y)$ actually gives a map from $C_t$ to $E_t$. This amounts to checking that \[f_2(x)^2(x (x+1) (x^{2g-1}+tj(x)^2)) = f_1(x)(f_1(x)+1)(f_1(x)+t)\] is a formal identity. This does happen to be the case; hence $(f_1(x),f_2(x)y)$ does define a map from $C_t$ to $E_t$.

Now, we check that the ramification type is as claimed. To do this, we observe that if $f(x,y)=(f_1(x),f_2(x)y)$ is the map above, and $\omega=\frac{dx}{y}\in\Omega^1_{E_t}$ is an invariant differential on $E_t$, then \[f^\ast\omega=(2g-1)\frac{x^{g-1}\; dx}{y}.\] So, $f^\ast\omega$ vanishes to order $2g-2$ at $(0,0)$ and has no other zeros or poles. Hence, $f$ is totally ramified at $(0,0)$ and unramified everywhere else. \end{proof}

\begin{rem} It is worth noting that the map $(f_1(x),f_2(x)y)$ is independent of $t$ and hence defines a map $F:\PP^2_{\CC}\to\PP^2_{\CC}$. If we fix an elliptic curve $E_t$ in the target $\PP^2$, then $F^{-1}(E_t)$ is a union of several irreducible components, one of which is $C_t$. If we take $t=0$ or $t=-1$, then $E_t$ is a nodal cubic, and $C_t$ is a singular quintic of arithmetic genus 0. This will be relevant in the next section. \end{rem}

\section{Degeneration techniques} \label{degen}

The proof given in the previous section thoroughly fails to capture the motivation that went into the discovery of this result. In fact, the story of finding these examples is much more interesting than is the proof. Therefore, we now discuss how the reader could (and the author did) discover such an example. To do this, we carefully work with the lowest-degree example: that of a degree-3 origami from a genus-2 curve to an elliptic curve. Such an origami must necessarily be totally ramified.

In the remainder of this section, we perform some educated guesswork; it will not be clear whether our guesses will turn out to be successful until we present a proof in the style of that of Theorem \ref{totramorigami}.

We will construct a family of genus-2 curves mapping to a family of elliptic curves, parametrized (essentially) by their Legendre form. Hence, for any $j$-invariant other than 0 or 1728, we will actually construct six genus-2 curves mapping to an elliptic curve with this $j$-invariant. These six genus-2 curves come in three pairs of isomorphic curves; hence, we generically obtain three pairwise nonisomorphic covers in this way.

\begin{figure}  \begin{center} \includegraphics[height=1in]{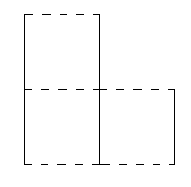} \end{center} \caption{Here, we shrink the dotted edges to a point. The resulting surface has geometric genus 0.} \label{Lorigamidegen} \end{figure} \begin{figure}  \begin{center} \includegraphics[height=1in]{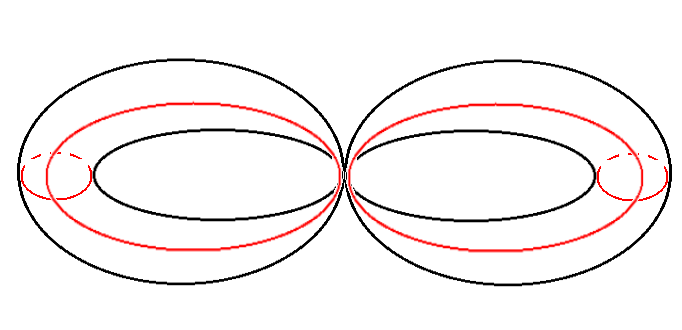} \end{center} \caption{A three-dimensional version of Figure \ref{Lorigamidegen}.} \label{threepinches} \end{figure}

%\begin{vchfigure} \includegraphics[height=1in]{L-origami-degen.png} \vchcaption{Here, we shrink the dotted edges to a point. The resulting surface has geometric genus 0.} \label{Lorigamidegen} \end{vchfigure} \begin{vchfigure} \includegraphics[height=1in]{threepinches.png} \vchcaption{A three-dimensional version of Figure \ref{Lorigamidegen}.} \label{threepinches} \end{vchfigure}

To do this, we start by constructing a cover $C'$ of the nodal cubic \[E':y^2=x^3+x^2,\] which we expect to arise as a degeneration of covers of elliptic curves which limit to $E'$. One possibility is that the degenerate origami diagram will look like Figure \ref{Lorigamidegen}, with the dotted edges collapsed to a point. The curve represented by this origami has geometric genus 0, since it is a double torus with three pinched loops, as in Figure \ref{threepinches}. Furthermore, since the origamis are totally ramified, the family of covers must degenerate to a curve with only one preimage of the branch point in $E'$. Finally, a map $C'\to E'$ can be described as a map from the normalization of $C'$ to the normalization of $E'$.

The next thing to do is to construct an explicit equation for $C'$, as well as its normalization map. While in general this is a notoriously difficult problem, it is easy in this case. By the picture, we can see that $C'$ has one nodal point and has geometric genus 0; hence it has a Weierstra\ss\ equation of the form $y^2=(x-a)^4(x-b)$. We choose to take $a=0$ and $b=-1$ so that we obtain the curve \[C':y^2=x^5+x^4.\]

To compute the normalization of $E'$, we note that the map $E'\to\PP^1$ given by $(x,y)\mapsto x+1$ has a square root $y/x$ in $\CC(E')$. Letting $u=y/x$, we have $x=u^2-1$ and $y=u^3-u$, so $\CC(E')=\CC(u)$, and the normalization map is $\PP^1\to E'$, given by $u\mapsto(u^2-1,u^3-u)$. A similar computation shows that the normalization of $C'$ is $\PP^1\to C'$, given by $t\mapsto (t^2-1,t(t^2-1)^2)$. Note that, in the normalizations of both $C'$ and $E'$, the preimage of the nodal point is $\{\pm 1\}\subset\PP^1$.

The map on normalizations must have the same degree as the map $C'\to E'$, and it can only be branched at the preimages of the node of $E'$ and the branch point of the map $C'\to E'$. In this case, the only possibilities are $\{\pm 1\}$, so by the Riemann-Hurwitz formula it must be branched at both points. Fortunately, there are very few maps $\PP^1\to\PP^1$ with only two branch points: they are simply conjugates of $z\mapsto z^n$, where $n$ is the degree of the map. In this case, the map on normalizations is \[z\mapsto\frac{z^3+3z}{3z^2+1}.\]

Now, in order to compute the map $f:C'\to E'$, note that we have the following commutative diagram: \[\xymatrix{\PP^1\ar[r]\ar[d] & \PP^1\ar[d] \\ C'\ar[r]^f & E'}\] Furthermore, the vertical maps have near-inverses; the inverse of the vertical arrow on the left is given by $(x,y)\mapsto y/x^2$. Hence, $f$ is the composition of the other three arrows; putting this together, we have \[f(x,y)=\left(\frac{x^3}{(3x+4)^2},\frac{xy(x+4)}{(3x+4)^3}\right).\]

We now proceed to prolong $f$ to a map from a family of genus-2 curves to the Legendre family of elliptic curves by means of deformations.

In order to figure out the map from a family of nonsingular genus-2 curves to a family of elliptic curves, we deform the defining equations for the nodal quintic and for the map. We let the defining equation of the genus-2 curve be \[C_t:y^2=x^5+(1+at)x^4+btx^3+ctx^2+dtx,\] where $a,b,c,d\in\CC[\![t]\!]$. The defining equation of the elliptic curve will be \[E_t:y^2=x(x+1)(x+t).\] The map will be \[(x,y)\mapsto\left(\frac{x^3}{((3+et)x+(4+ft))^2},\frac{(x^2+(4+gt)x)y}{((3+et)x+(4+ft))^3}\right).\] A priori, $a,b,c,d,e,f,g$ are power series in $t$; for now, we are only interested in their constant terms. Expanding everything out and equating the $tx^i$ terms for various values of $i$ gives us a system of linear equations; we then find that \begin{align*} a&= 9 \\ b&= 33 \\ c&= 40 \\ d&=16 \\ e=f=g&= 0 \end{align*} is a solution. In fact, these values of $a,b,c,d,e,f,g$ are not merely the constant terms of power series; they are in fact the entire power series. Hence, if we let \[C_t:y^2=x^5+(1+9t)x^4+33tx^3+40tx^2+16tx\] and \[E_t:y^2=x(x+1)(x+t),\] then \[f(x,y)=\left(\frac{x^3}{(3x+4)^2},\frac{xy(x+4)}{(3x+4)^3}\right)\] is a map $f:C_t\to E_t$. Indeed, this map is only branched over $(0,0)$, with its preimage being $(0,0)$; we can check this directly, or we can verify that the pullback of the invariant differential $\omega=\frac{dx}{y}\in\Omega^1_{E_t}$ (which has no zeros or poles) is $3x\frac{dx}{y}\in\Omega^1_{C_t}$, which has a double zero at $(0,0)$ and no other zeros or poles.

The same method allows us to construct totally ramified origamis in every genus. For instance, in genus 3, the degenerate curve has the form $y^2=x^6(x+1)$, and the origami diagram is shown in Figure \ref{orig3}. In general, the origami diagrams we use for the constructions here are staircase-shaped.

\begin{center}\begin{figure}  \includegraphics[height=1.5in]{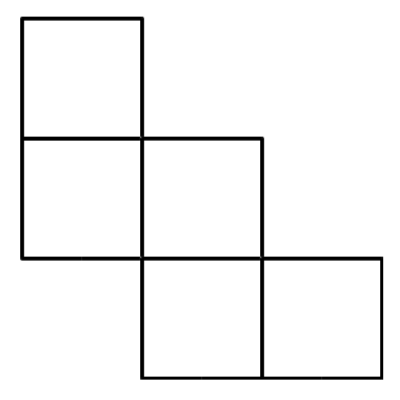} \caption{A genus-3 origami. Here, opposite sides are glued together.} \label{orig3} \end{figure}\end{center}

When we perform a degeneration procedure as we did in the genus-2 case, we find that \[C_t:y^2=x^7+(1+25t)x^6+225tx^5+760tx^4+1200tx^3+896tx^2+256tx\] and \[E_t:y^2=x(x+1)(x+t).\] Then \[f(x,y)=\left(\frac{x^5}{(5x^2+20x+16)^2},\frac{x^3(x^2+12x+16)y}{(5x^2+20x+16)^3}\right)\] is a totally ramified origami $f:C_t\to E_t$.

It is worth noticing that, in these cases, we need only change the equation of the genus-$g$ curve as $t$ varies; in particular, the map does not change. Hence, we have a map $f:\PP^2\to\PP^2$ so that the inverse images of elliptic curves in a certain family are all genus-$g$ curves, so that the map is an origami. The proof above explains this phenomenon.

It would be interesting to see this method generalize to cases where the base need not be an elliptic curve. In particular, we would like to know to what extent is it possible to construct branched covers of a curve $C$ in $\PP^2$ by constructing a suitable map $f:\PP^2\to\PP^2$, chosen so that its branch locus is consistent with the desired branching properties of the cover of $C$, and restricting to the map $f\mid_D:D\to C$, where $D$ is some irreducible component inside $f^{-1}(C)$ for which $f\mid_D:D\to C$ is flat. The author has used this method to construct several examples of branched and unbranched covers of higher-genus curves, but a detailed study of this method may be the topic of future work.

It is not so easy to use the techniques in this article to write down curves $C$ and maps $f$ for other ramification types. One challenge is that it is unclear how the degenerate pictures ought to look. Another challenge is that, even if it were clear, sometimes the normalizations of the covering degenerate curve may have positive genus, and we would then need to write down explicit equations for maps from a positive-genus curve to a genus-1 curve, and there is no known good procedure for doing so. The author has developed different techniques one can use to write down equations with other branching types and some of these techniques are presented in \cite{RS12}.

\section*{Acknowledgements}
The author would like to thank Akshay Venkatesh for suggesting this problem and for many helpful discussions and comments.

%% \tableofcontents  % Produces the table of contents.

\providecommand{\WileyBibTextsc}{}
\let\textsc\WileyBibTextsc
\providecommand{\othercit}{}
\providecommand{\jr}[1]{#1}
\providecommand{\etal}{~et~al.}

%\bibliographystyle{mn}
%\bibliography{totram}

\end{document}